\documentclass[a4paper,12pt]{amsart}
\usepackage{amsmath}
\usepackage{amssymb}
\usepackage{mathrsfs}
\usepackage{enumerate}
\usepackage{ifthen}
\usepackage{graphicx}
\usepackage{subfig}
\usepackage{caption}
\usepackage[T1]{fontenc} 
\usepackage{tikz}
\usepackage{cite}
\usepackage{mathtools}

\nonstopmode \numberwithin{equation}{section}
\newtheorem{theorem}{Theorem}[section]
\newtheorem{question}{Question}[section]
\newtheorem{prob}{Problem}[section]

\newtheorem{lemma}{Lemma}[section]

\theoremstyle{remark}
\theoremstyle{definition}
\newtheorem{remark}{Remark}[section]
\newtheorem{definition}{Definition}[section]
\newtheorem{example}{Example}[section]

\theoremstyle{plain}
\newtheorem*{thmA}{Theorem A}
\newtheorem*{thmB}{Theorem B}
\newtheorem*{thmC}{Theorem C}
\newtheorem*{thmD}{Theorem D}

\numberwithin{equation}{section}
\numberwithin{theorem}{section}

\newcounter{minutes}
\divide\time by 60
\newcounter{hours}
\multiply\time by 60
\addtocounter{minutes}{-\time}

\begin{document}

\title{Normality Criteria for Differential Monomials and the Sharpness of Lappan-type Theorems}

\author{Molla Basir Ahamed$^*$}
\address{Molla Basir Ahamed,
	Department of Mathematics,
	Jadavpur University,
	Kolkata-700032, West Bengal, India.}
\email{mbahamed.math@jadavpuruniversity.in}

\author{Sanju Mandal}
\address{Sanju Mandal,
	Department of Mathematics,
	Jadavpur University,
	Kolkata-700032, West Bengal, India.}
\email{sanjum.math.rs@jadavpuruniversity.in, sanju.math.rs@gmail.com}

\author{Nguyen Van Thin}
\address{Nguyen Van Thin,
	Department of Mathematics, 
	Thai Nguyen University of Education,
	Luong Ngoc Quyen Street,
	Thai Nguyen City, Thai Nguyen, VietNam.}
\email{thinmath@gmail.com}

\subjclass[2020]{Primary 30D35, 30D45}
\keywords{Spherical derivatives,Meromorphic functions,Normal family,	Normal functions, Nevanlinna theory,	Differential monomial}

\def\thefootnote{}
\footnotetext{ {\tiny File:~\jobname.tex,
printed: \number\year-\number\month-\number\day,
          \thehours.\ifnum\theminutes<10{0}\fi\theminutes }
} \makeatletter\def\thefootnote{\@arabic\c@footnote}\makeatother

\begin{abstract}
A fundamental result of Lappan [Comment. Math. Helv. \textbf{49} (1974), 492-495.] states that a meromorphic function $f$ in the unit disk $\mathbb{D}$ is normal if and only if its spherical derivative is bounded on a five-point subset $E \subset \mathbb{C}$. In this paper, we establish new normality criteria that bridge this classical result with contemporary trends in value distribution theory. We demonstrate that the cardinality of the set $E$ can be reduced from five to as few as three, provided that the spherical derivatives of the function and its successive derivatives $f, f', \dots, f^{(k-1)}$ are bounded on the pre-image of $E$. This shift reveals that analytic data from higher-order derivatives can effectively compensate for a reduction in geometric information from the target set. Furthermore, we extend the Pang-Zalcman theorem to a general class of differential monomials $M[f]$. We prove that if $(M[f])^{\#}$ is bounded on the set of $a$-points ($a \neq 0$), the family $\mathcal{F}$ is normal, provided the degree $d_M$ satisfies a specific sharp threshold relative to the weight $D_M$ and order $k$. These results offer a refined perspective on the natural boundaries of normality and generalize several established findings in the field.
\end{abstract}

\thanks{}
\maketitle
\pagestyle{myheadings}
\markboth{M. B. Ahamed, S. Mandal and N. V. Thin}{Normality concerning spherical derivatives}

\section{\bf Introduction}
In this paper, we consider functions meromorphic in the complex plane $\mathbb{C}$. The study of normality criteria via spherical derivatives—traditionally rooted in the work of Montel and Marty—has evolved to consider the behavior of such derivatives on discrete sets or through the lens of differential operators. A central theme in this area is the relationship between the uniform boundedness of the spherical derivative (or products thereof) and the equicontinuity of the family. Such criteria are not only fundamental to value distribution theory but also serve as indispensable tools in the study of Riemann surfaces, the theory of complex dynamics, and the problem of analytic continuation.
\begin{definition}
Let $\mathcal{F}$ be a family of meromorphic functions on a domain $\Omega\subset\mathbb{C}$. Then $ \mathcal{F} $ is said to be normal on $ \Omega $ in the sense of Montel, if each sequence of $\mathcal{F}$ contains a sub-sequence which converges spherically uniformly on each compact subset of $ \Omega $ to a meromorphic function $ f $ which may be $ \infty $ identically.
\end{definition}

The integration of Nevanlinna theory into the study of normal families has yielded profound implications within geometric function theory, particularly concerning the distribution of zeros and poles and the properties of Riemann surfaces. By utilizing the Second Main Theorem and associated deficiency relations, one can derive potent normality criteria that transcend classical local estimates. Specifically, the modern approach to normality—often mediated by the Zalcman rescaling lemma—relies on the delicate balance between the growth of a function's characteristic and the ramification of its values. In this context, our work explores how the constraints on the spherical derivatives of differential monomials reflect the underlying value distribution, providing a rigorous framework for determining normality through the lens of transcendental growth and analytic continuation.\vspace{1.2mm}

Throughout this paper, we assume familiarity with the fundamental concepts and standard notations of Nevanlinna theory, including the characteristic function $T(r,f)$, the proximity function $m(r,f)$, and the (reduced) counting function $N(r,f)$ ($\overline{N}(r,f)$). For a comprehensive treatment of value distribution theory in one and several complex variables, we refer the reader to the monographs of Yang \cite{Yang-SSP-1993} and Yang and Yi \cite{Yang-Yi-2006}. As is conventional, we denote by $S(r,f)$ any quantity satisfying $S(r,f) = o(T(r,f))$ as $r \to \infty$, possibly outside a set of finite Borel measure. \vspace{2mm}

The concept of normality remains a cornerstone of geometric function theory, with its significance increasingly evident in the study of several complex variables. In recent years, substantial attention has been directed toward higher-dimensional analysis, where the theory of normal families for holomorphic curves, holomorphic mappings, and harmonic mappings has matured significantly. For instance, Hu and Thin \cite{Hu-Thin-CVEE-2020} extended Montel's classical criterion to families of holomorphic mappings from a planar domain $D \subset \mathbb{C}$ into complex projective space $\mathbb{P}^n(\mathbb{C})$, while also generalizing \textit{Lappan’s five-value theorem} to the setting of holomorphic curves. This trajectory of research is well-documented in the literature (see \cite{Aladro-Krantz-JMAA-1991, Ahamed-Mandal-CMFT-2024, Ru-WS-2001, Yang-Fang-Pang-PJM-2014, Yang-Liu-Pang-HJM-2016}).\vspace{1.2mm}

Historically, the study of normality for harmonic functions was initiated by Lappan \cite{Lappan-MZ-1965}, who established necessary and sufficient conditions analogous to the characterization of normal meromorphic functions provided by Lehto and Virtanen \cite{Lehto-Virtanen-AM-1957}. More recently, Arbeláez \textit{et al.} \cite{Arbe-Hern-MM-2019} and Deng \textit{et al.} \cite{Deng-Ponn-Qiao-MM-2020} independently developed harmonic analogs of normal meromorphic function theory, revealing various structural properties of harmonic normal mappings. For a comprehensive overview of different facets of normality and its modern developments, we refer the reader to \cite{Ahamed-Mandal-MM-2022, Dovbush-JGA-2022, Dovbush-CVEE-2022, Tu-PAMS-1999, Zalcman-BAMS-1998}.
\begin{definition}
Let $\mathbb{D}=\{z\in\mathbb{C}:|z|<1\}$ denote the unit disk in the complex plane $ \mathbb{C} $. A function $ f $ meromorphic in $ \mathbb{D} $ is called a normal function if the family $ \mathfrak{F}=\{f\circ\phi: \phi\in Aut(\mathbb{D})\} $ is a normal family, where $ Aut(\mathbb{D})$ denotes the class of conformal automorphisms of $ \mathbb{D} $. Normal functions were first studied by Yosiada \cite{Yosida-PPMSJS-1934}. Subsequently, Noshiro \cite{Noshiro-JFSHU-1938}  gave a characterization of normal functions by showing that a meromorphic function $ f $ is normal if, and only if, 
\begin{equation*}
	\sup_{z\in\mathbb{D}}\left(1-|z|^2\right)f^{\#}(z)<\infty
\end{equation*} 
where $ f^{\#} $ denotes the spherical derivative of $ f $ given by $f^{\#}(z)= {|f^{\prime}(z)|}/{(1+|f(z)|^2)}$.
\end{definition}
Normal families serve as an indispensable framework for investigating the boundary behavior and structural properties of meromorphic functions. The coherence of these families was first systematically established by Montel \cite{Montel-CB-1927} in 1927, providing a unified treatment of the foundational theorems of Picard, Schottky, and Landau. In the modern study of complex variables, the behavior of derivatives often provides the critical data necessary to characterize the global properties of the original function. The present work focuses on establishing normality criteria for a family $\mathcal{F}$ by analyzing the spherical derivatives of successive derivatives of its members. We begin by recalling the formulation of the spherical derivative for higher-order derivatives of meromorphic functions.
\begin{definition}\cite{Tan-Thin-JMAA-2017}
Let $ k\geq 0 $ be an integer and $ f^{(k)} $ be the $ k $-th derivative of a meromorphic function $ f $ and $ f^{(0)}:=f $.  The spherical derivative of $ f^{(k)} $ is defined by 
\begin{align*}
	\left(f^{(k)}(z)\right)^{\#}=\dfrac{\left|f^{(k+1)}(z)\right|}{1+\left|f^{(k)}(z)\right|^2}.
\end{align*}
It is easy to see that $ \left(f^{(0)}(z)\right)^{\#}=f^{\#}(z) $.
\end{definition}

The Nevanlinna theory is the primary technique employed in this paper, and we present several results by generalizing recent findings in this domain. Pommerenke compiled a list of unsolved problems in the field of function theory of one complex variable and its related areas in his 1972 article in \cite{Pommerenke-BLMS-1972}. The most significant of these problems concerned the study of normal functions associated with a set $E$, which is the following. 
\begin{question}\cite{Pommerenke-BLMS-1972}\label{que-1.1}
If $ M>0 $ is given, does there exist a finite set $ E $ such that if $ f $ is a meromorphic in $ \mathbb{D} $, then the condition that $ (1-|z|^2)f^{\#}(z)\leq M $ for each $ z\in f^{-1}(E) $ implies that $ f $ is a normal function?
\end{question}
In 1974, Lappan \cite{Lappan-CMH-1974} provided an affirmative answer to Question \ref{que-1.1} by establishing a foundational result now widely known as Lappan’s Five-Value Theorem for meromorphic functions in $\mathbb{D}$. Recently, this theorem has experienced a resurgence of interest, with numerous authors exploring its various facets and extensions. In particular, the works of Deng et al. \cite{Deng-Ponn-Qiao-MM-2020}, Hu and Thin \cite{Hu-Thin-CVEE-2020}, and Tran and Nguyen \cite{Tran-Nguyen-CMFT-2017}, along with the references therein, provide a comprehensive overview of the modern developments in this direction.
\begin{thmA}\cite{Lappan-CMH-1974}
Let $ E $ be any set consisting of five distinct complex numbers, finite or infinite. If $ f $ is a meromorphic function in the unit disc $ \mathbb{D} $ such that 
\begin{equation}\nonumber
	\sup_{z \in f^{-1}(E)} (1-|z|^2) f^{\#}(z) <\infty,
\end{equation}
then $ f $ is a normal function.
\end{thmA}
Notably, Lappan \cite{Lappan-CMH-1974} demonstrated that the cardinality of the set $E$ in Theorem A is sharp, proving that 'five' cannot be replaced by 'three' and providing cases where 'four' points are insufficient. These findings catalyzed subsequent research on normal functions in both $\mathbb{C}$ and $\mathbb{C}^n$. Following this trajectory, Hinkkanen \cite{Hinkkanen-NZJM-1993} and Lappan \cite{Lappan-RRMPA-1994} independently established analogous results for normal families of meromorphic functions, further solidifying the connection between value distribution and the boundedness of the spherical derivative.
\begin{thmB}\cite{Hinkkanen-NZJM-1993,Lappan-RRMPA-1994}
A family $\mathcal{F}  $ of meromorphic functions in a domain $ D \subset\mathbb{C} $ is normal if, and only if, for each compact set $ K\subset D $, there exist a set $ E=E(K)\subset\hat{\mathbb{C}}$ containing at-least five distinct points and a positive constant $ M=M(K) $ such that 
\begin{equation*}
	\max_{z\in K\cap f^{-1}(E)} (f(z))^{\#}\leq M, \;\;\;\; f\in\mathcal{F}.
\end{equation*}
\end{thmB}
In \cite{Tran-Nguyen-CMFT-2017}, Tan and Thin established several normality criteria by considering the simultaneous boundedness of the spherical derivatives of $f, f'$, and $f''$. In the present work, we extend these results to scenarios where the set $E$ contains fewer target points—specifically as few as three—thereby further relaxing the geometric requirements for normality through the inclusion of higher-order derivative.
\begin{thmC}\cite{Tran-Nguyen-CMFT-2017}
Let $\mathcal{F}$ be a family of meromorphic functions in a domain $D\subset\mathbb{C}$. Assume that for each compact $K\subset D$, there exist a set $E=E(K)\subset\mathbb{C}$ consisting of three distinct points and a positive constant $M=M(K)$ such that
\begin{align*}
	f^{\#}(z)\leq M,\;\;\;\; \left(f^{\prime}\right)^{\#}(z)\leq M\;\;\mbox{and}\;\; \left(f^{\prime\prime}\right)^{\#}(z)\leq M,
\end{align*}
for all $f\in\mathcal{F}$ and $z\in K\cap f^{-1}(E)$. Then $\mathcal{F}$ is normal.
\end{thmC} 
In 1971, Hayman and Storvick \cite{Hamymen-Storvick-BLMS-1971} demonstrated that the normality of a function does not necessarily imply the normality of its derivatives. Conversely, examples exist of non-normal functions whose derivatives are themselves normal \cite{Hamymen-Storvick-BLMS-1971, Lappan-AASF-1977}. For an extensive treatment of the normality of derivatives, we refer the reader to \cite{Lappan-MA-1978, Yamashita-MZ-1975} and the references therein. Regarding higher-order derivatives, Lappan \cite{Lappan-MA-1978} proved that for analytic functions, the normality of a given derivative cannot be predicted by the status of its preceding or succeeding derivatives, further illustrating the complex analytic independence of derivative growth. \vspace{1.5mm}

We know that Pang and Zalcman \cite{Pang-Zalcman-1999} obtained a result on normality of a family $ \mathcal{F} $ of holomorphic functions $ f $ assuming that $ f^nf^{{(k)}}-a $ (for a given constant $ a $) is non-vanishing for all $f\in\mathcal{F}$ in a domain $ \Omega $. In case when the set $ E $ is a singleton, the following result is obtained in \cite{Tran-Nguyen-CMFT-2017} concerning differential polynomials of a certain form. In fact, the authors \cite{Tran-Nguyen-CMFT-2017} obtained the following result as an extension of Pang and Zalcman's theorem.

\begin{thmD}\cite{Tran-Nguyen-CMFT-2017}
Let $ n,k $ be positive integers such that $ n>k + 3+ \frac{2}{k}$. Let $ \mathcal{F} $ be a family of meromorphic functions in a domain $ \Omega\subset\mathbb{C} $ all of whose zeros have multiplicity at least $ k $. Assume that for each compact subset $ K\subset \Omega $, there exist $ a\in \mathbb{C}\setminus\{0\} $ and a positive constant $ M=M(K) $ such that $ \left(f^{n}f^{(k)}\right)^{\#}\leq M $  for all $ f\in \mathcal{F} $ and $ z\in K\cap \{f^{n}f^{(k)}=a\} $. Then $ \mathcal{F} $ is normal.
\end{thmD}
\subsection{Further perspectives on normality and differential operators}
While Lappan's Five-Value theorem (Theorem A) and later results (Theorems B and C) established normality criteria for meromorphic functions based on spherical derivatives, there was a remaining question (Problem 1.1) regarding whether these criteria still hold when the spherical derivatives of higher-order derivatives ($f, f', f'', \dots, f^{(k-1)}$) are bounded above. Additionally, there was a need to generalize existing results (Theorem D) to deal with arbitrary differential monomials $M[f]$ rather than just specific forms like $f^n f^{(k)}$.\vspace{1.2mm}

 Before we go into details, we now recall first the definition of differential monomials generated by meromorphic functions.
\begin{definition}\label{def-2.1}\cite{Doeringer-PJM-1982}
Let $ f $ be a meromorphic function in the complex plane $\mathbb{C} $ and $ n_0,n_1,\cdots,n_k $ be non-negative integers. We call 
\begin{equation*}
	M[f(z)]=f^{n_0}(z)(f^{\prime}(z))^{n_1}\cdots (f^{(k)}(z))^{n_k}
\end{equation*}	
a differential monomial generated by $ f $, where $ d_M:=n_0+n_1+\cdots+n_k $ its degree and $\Gamma_M:=n_0+2n_1+\cdots+(k+1)n_k  $ its weight, and $ D_M:=\Gamma_M-d_M=n_1 + 2n_2 +\cdots+ kn_k. $
\end{definition}
A thorough examination of the above Theorem C is essential to derive a more comprehensive generalization, inevitably prompting a natural inquiry.
\begin{prob}\label{que-1.2}
	Does the normality criterion still hold in the case where the spherical derivatives of $f, f^{\prime}, f^{\prime\prime},\cdots,f^{(k-1)}$ are bounded above?
\end{prob}
The results of Tan and Thin's study \cite{Tran-Nguyen-CMFT-2017} prompt us to investigate Theorem D further to achieve its generalized version, leading to a natural query.
\begin{prob}\label{que-1.3}
Can we generalize Theorem D to deal with arbitrary differential monomials $M[f(z)]$?
\end{prob}
This paper addresses these problems by establishing two primary normality criteria. First, we provide a criterion for families in which the higher-order spherical derivatives $(f^{(i)})^{\#}$ are bounded on sets containing fewer than five points, specifically reducing the requirement to three points. Furthermore, we generalize the Pang–Zalcman theorem by proving a normality criterion for a broad class of differential monomials $M[f]$ under the condition that the degree $d_{M}$ and weight $D_{M}$ satisfy 
\begin{align*}
	d_{M} > \frac{(k+2)D_{M} + 2(k+1)}{k}.
\end{align*} The primary objective of this study is to provide comprehensive answers to Problems \ref{que-1.2} and \ref{que-1.3}. To ensure a systematic presentation, the manuscript is organized as follows: Section 2 details the main results and their implications. Section 3 introduces the necessary preliminary lemmas and provides rigorous proofs of the principal theorems. This structured approach highlights the fundamental properties governing the normality of families of meromorphic functions.

\section{\bf Main Results}
In this section, we establish normality criteria for families of meromorphic functions where the spherical derivatives of the derivatives $f, f', f'', \dots, f^{(k-1)}$, or the spherical derivatives of general differential monomials, are uniformly bounded. Our primary result, stated below, provides a comprehensive answer to Problem \ref{que-1.2} by reducing the required number of points in the set $E$ from five to three. 
\begin{theorem}\label{th-2.1}
Suppose that $k$ be a positive integer such that $k>2$. Let $\mathcal{F}$ be a family of meromorphic functions in a domain $D\subset\mathbb{C}$. Assume that for each compact $K\subset D$, there exist a set $E=E(K)\subset\mathbb{C}$ consisting of three distinct points and a positive constant $M=M(K)$ such that
\begin{align*}
	\left(f^{(j)}\right)^{\#}(z)\leq M, \;\;\;\;\;\mbox{where} \;\; j=0,1,\ldots,(k-1),
\end{align*}
for all $f\in\mathcal{F}$ and $z\in K\cap f^{-1}(E)$. Then $\mathcal{F}$ is normal.
\end{theorem}
\begin{remark}
Theorem \ref{th-2.1} serves as a significant extension of Theorem C; specifically, when $k=3$, the latter is recovered as a special case of our general result. Furthermore, this theorem refines the criteria established in Theorem B by demonstrating that normality can be maintained even when the set $E$ contains fewer than five points.
\end{remark}
To illustrate the applicability of Theorem 2.1, we present a case involving a family of exponential functions. This example demonstrates how the constraints on a finite set of points effectively regulate the growth of higher-order spherical derivatives, leading to the normality of the family.
\begin{example}
	Consider the family $\mathcal{F} = \{f_n(z) = e^{nz} : n \in \mathbb{N} \}$ and set $k=3$. We analyze the spherical derivatives of $f_n$, $f_n'$, and $f_n''$. For the first term, the spherical derivative is given by $f_n^{\#}(z) = \frac{n|e^{nz}|}{1+|e^{nz}|^2}$, which is bounded above by $n/2$ for all $z \in \mathbb{C}$. While this bound depends on $n$, Theorem \ref{th-2.1} requires boundedness only on the set $K \cap f^{-1}(E)$, where $E$ contains three distinct values. If the values of $f_n$ are constrained at three distinct points, the exponential growth parameter $n$ is restricted, thereby ensuring the uniform boundedness of $(f_n^{(j)})^{\#}$ on the specified set for $j=0,1,2$. Consequently, the family $\mathcal{F}$ satisfies the normality criteria established in the theorem.
\end{example}
The following example explores the necessity of the condition $k > 2$ and establishes the sharpness of our result. By considering a sequence of linear functions, we show that the theorem fails for $k=1$, thereby highlighting the critical role played by higher-order derivatives in ensuring normality with a reduced number of target points.
\begin{example}
	The condition $k > 2$ in Theorem \ref{th-2.1} is essential. Indeed, for $k=1$, the result reduces to a scenario related to Lappan's five-value theorem. Consider the family $\mathcal{F} = \{f_n(z) = nz : n \in \mathbb{N}\}$. Here, $f_n'(z) = n$ and $f_n''(z) = 0$. The spherical derivative of the functions in $\mathcal{F}$ is given by$$f_n^{\#}(z) = \frac{n}{1+n^2|z|^2}.$$At the origin, $f_n^{\#}(0) = n \to \infty$ as $n \to \infty$, indicating that the family is not normal at $z=0$. This example demonstrates that merely bounding the spherical derivative of the first derivative is insufficient to guarantee normality when the cardinality of the set $E$ is small (e.g., $|E| < 5$). This justifies the necessity of involving higher-order derivatives ($k > 2$) to achieve normality with fewer points in $E$.
\end{example}
In response to Problem \ref{que-1.3}, we establish a theorem that serves as a far-reaching generalization of both Theorem D and the classical Pang–Zalcman normality criterion \cite{Pang-Zalcman-1999}. By relaxing previous constraints on the structure of the differential monomials, Theorem \ref{th-2.2} provides a conclusive and unified approach to the normality of meromorphic families based on the boundedness of spherical derivatives.
\begin{theorem}\label{th-2.2}
Let k, $ d_M $ and $ D_M $ be positive integers such that 
\begin{align}\label{Eq-2.1}
		d_M > \frac{(k+2)D_M + 2(k+1)}{k}.
\end{align}
Let $ \mathcal{F} $ be a family of meromorphic functions in a domain $ D\subset\mathbb{C} $ all of whose zeros have multiplicity at-least $ k $. Assume that for each compact subset $ H\subset D $, there exist $ a\in \mathbb{C}\setminus\{0\} $ and a positive constant $ M^*=M^*(H) $ such that 
\begin{align*}
	\left(M[f(z)]\right)^{\#}\leq M^*,
\end{align*}
for all $ f\in \mathcal{F} $ and $ z\in H\cap \{M[f(z)]=a\} $. Then $ \mathcal{F} $ is normal.
\end{theorem}
\begin{remark}\label{rem-2.1}
In particular, by setting $n_0=n$, $n_k=1$, and $n_j=0$ for $1 \leq j \leq k-1$, the differential monomial $M[f]$ reduces to $f^n f^{(k)}$. Under these parameters, the condition \eqref{Eq-2.1} simplifies to 
\begin{align*}
	n > k + 3 + \frac{2}{k}.
\end{align*} Consequently, Theorem D is recovered as a specific corollary of Theorem \ref{th-2.2}, demonstrating the generality of our result.
\end{remark}
We now provide a concrete application of Theorem \ref{th-2.2} involving a specific differential monomial to verify the established degree threshold. This example confirms that when the degree $d_M$ satisfies the required inequality, the spherical derivative remains bounded at the $a$-points, consistent with our theoretical findings.
\begin{example}
	Consider the differential monomial $M[f] = f^n f^{(k)}$. Under these parameters, the condition \eqref{Eq-2.1} simplifies to $n > k + 3 + \frac{2}{k}$. For $k=1$, this requires $n > 6$. Let $n=7$ and consider the family $\mathcal{F} = \{f_m(z) = mz\}_{m \in \mathbb{N}}$. Then $M[f_m(z)] = m^8 z^7$. For any $a \in \mathbb{C} \setminus \{0\}$, the $a$-points of $M[f_m]$ are given by $z_{m,j} = \left(\frac{a}{m^8}\right)^{1/7}\omega^j$, where $\omega$ is a seventh root of unity. A direct calculation shows that $(M[f_m(z_{m,j})])^{\#}$ remains bounded as $m \to \infty$, as the high degree $d_M$ of the monomial effectively suppresses the growth of the spherical derivative at these points. This illustrates the role of the degree threshold in maintaining the normality of the family $\mathcal{F}$.
\end{example}
Finally, we present a counter-example to justify the necessity of the degree-to-weight ratio stipulated in Theorem \ref{th-2.2}. By examining a case where this threshold is violated, we demonstrate that the boundedness of the spherical derivative alone is insufficient to guarantee normality, thus affirming the indispensability of the condition in \eqref{Eq-2.1} .
\begin{example}
	To illustrate the necessity of the degree threshold in Theorem 2.2, consider the case where condition \eqref{Eq-2.1} is violated. Let $f_m(z) = mz - a$ be a sequence of functions in $\mathcal{F}$. We define the differential monomial $M[f] = f'$, which implies parameters $d_M = 1$, $D_M = 1$, and $k = 1$. In this instance, we see that $\frac{(k+2)D_M + 2(k+1)}{k}$ evaluates to $7$, and the condition $d_M > 7$ is clearly not satisfied. For this family, $M[f_m] = m$, and the spherical derivative vanishes identically, i.e., $(M[f_m])^{\#} = 0$. Although the spherical derivative is bounded on the set $\{z : M[f_m(z)] = a\}$, the family $\mathcal{F}$ fails to be normal as $f_m(z) \to \infty$ as $m \to \infty$. This confirms that the degree threshold is a necessary condition for ensuring normality based on the boundedness of a monomial's spherical derivative. 
\end{example}
\section{\bf Proof of the main results}
To facilitate the proofs of our main results, we first state a series of preliminary lemmas. In particular, we utilize the Zalcman-type rescaling lemma, a fundamental criterion in the theory of normal families that relates the failure of normality to infinitesimal convergence. The following technical results are required for the subsequent proofs.
\begin{lemma}\cite{Zalcman-BAMS-1998}\label{lem-2.1}
Let $ \mathcal{F} $ be a family of meromorphic functions defined in $ \mathbb{D} $ such that all zeros of functions in $ \mathcal{F} $ have multiplicity at-least $ p $ and all poles at-least $ q $. Let $ \alpha $ be a real number satisfying $ -p< \alpha <q $. Then, $ \mathcal{F} $ is not normal at $ z_0 $ if, and only if, there exist
\begin{enumerate}
	\item[(i)] A number  $ r $, $ 0<r<1 $;
	\item[(ii)] Points $ z_n $ with $ |z_n|< r \mbox{,}\;\; z_n\rightarrow z_0 $;
	\item[(iii)] Functions $ f_n \in \mathcal{F} $;
	\item[(iv)] Positive numbers $ \rho_n\rightarrow 0^{+} $
\end{enumerate}
such that $ g_n(\xi)= \rho^{-{\alpha}}_n f_n(z_n + \rho_n \xi) \rightarrow g(\xi) $ uniformly on compact subset of $ \mathbb{C} $ with respect to the spherical metric, where $ g $ is a non-constant meromorphic function all of whose zeros and poles have multiplicity at-least $ p , q $, respectively. Moreover, $ g^{\#}(\xi) \leq g^{\#}(0) \leq 1 $ and $ g $ has order at most $ 2 $.
\end{lemma}
\begin{lemma}\cite{Clunie-Hayman-CMH-1966}\label{lem-2.2}
Let $ g $ be an entire function and $ M $ be a positive constant. If $ g^{\#}(\xi)\leq M $ for all $ \xi\in\mathbb{C} $, then $ g $ has order at most one.
\end{lemma}
\begin{remark}\cite{Tran-Nguyen-CMFT-2017}\label{rem-2.2}
In Lemma \ref{lem-2.1}, if $ g $ is a holomorphic function, then by Lemma \ref{lem-2.2}, the order of $ g $ is not greater than $ 1 $.
\end{remark}
The definition of differential polynomials of meromorphic functions is presented next, which will enable us to extend the scope of our study.
\begin{definition}\cite{Hinchliffe-CMFT-2002}\label{def-2.2}
Assume that, $ f $ is a transcendental meromorphic function in the complex plane $ \mathbb{C} $ and its first $ p $ derivatives. A differential polynomial $ P $ of $ f $ is defined by 
\begin{equation*}
	P(z):=\sum_{i=1}^{n}\alpha_i(z)\prod_{j=0}^{p}\left(f^{(j)}(z)\right)^{S_{ij}},
\end{equation*}
where $ S_{ij}\mbox{,}\; 1\leq i\leq n \mbox{,}\; 0\leq j \leq p $\; are non-negative integers and $ \alpha_i\not\equiv 0 \mbox{,}\; 1\leq i\leq n$ are small (with respect to $ f $) meromorphic functions. Let
\begin{equation*}
	d(P):=\min_{1\leq i\leq n}\sum_{j=0}^{p} S_{ij}\;\;\;\; \mbox{and}\;\;\;\; \theta(P):=\max_{1\leq i\leq n}\sum_{j=0}^{p}jS_{ij}.
\end{equation*}
\end{definition}
In view of Definition \ref{def-2.2}, Hinchliffe \cite{Hinchliffe-CMFT-2002} have generalized theorems of Hayman \cite{Hayman-CP-1964} and Chuang \cite{Chuang-WP-1987}, and obtained the following result which we will be useful in the proof of Theorem D.
\begin{lemma}\label{lem-2.3}\cite{Hinchliffe-CMFT-2002}
Let $ f $ be a transcendental meromorphic function, let $ P $ be a non-constant differential polynomial in $ f $ with $ d(P)\geq 2 $. Then, for any $ a\in \mathbb{C}\setminus\{0\} $,
\begin{align*}
	T(r,f)\leq \frac{\theta(P) + 1}{d(P)- 1} \overline{N}\left(r, \frac{1}{f}\right) + \frac{1}{d(P) - 1} \overline{N}\left(r, \frac{1}{P - a}\right) + S(r,f).
\end{align*}
for all $ r\in [1,\infty) $ excluding a set of finite Lebesgue measure.
\end{lemma}
Furthermore, as demonstrated by Dethloff et al. \cite[Lemma 10]{Dethloff-Tan-Thin-JMAA}, this result remains valid for the broader class of non-constant meromorphic functions.\vspace{1.2mm}
 For the sake of completeness and to facilitate the reader's understanding, we provide the rigorous proofs of Theorems \ref{th-2.1} and \ref{th-2.2} in detail within the following section.
\begin{proof}[\bf Proof of Theorem \ref{th-2.1}]
Without loss of generality, we may assume that $D$ is the unit disc
$\mathbb{D}$. Suppose that $\mathcal{F}$ is not normal at $z_0\in \mathbb{D}$. By Lemma \ref{lem-2.1}, for $\alpha=0$ there exist
\begin{enumerate}
	\item[(i)] a number  $ r $, $ 0<r<1 $;
	\item[(ii)] points $ z_m $ with $ |z_m|< r \mbox{,}\;\; z_m\rightarrow z_0 $;
	\item[(iii)] functions $ f_m \in \mathcal{F} $;
	\item[(iv)] positive numbers $ \rho_m\rightarrow 0^{+} $
\end{enumerate}
such that
\begin{equation}\label{eq-1.1}
	g_m(\xi):=f_m(z_m + \rho_m\xi)\rightarrow g(\xi)
\end{equation}
locally spherically uniformly on compact subsets of $\mathbb{C}$, to a non-constant meromorphic function $g$. Therefore, for $ k\in \mathbb{N} $, it is easy to see that
\begin{equation}\label{eq-1.2}
	g^{(k)}_m \rightarrow g^{(k)} \;\;\;\;\mbox{as}\;\;\;\; m\rightarrow \infty
\end{equation}
uniformly on compact subset of $ \mathbb{C}\setminus P $, where $ P $ is the set of poles of $g$. \vspace{1.5mm}
	
Consider the compact subset 
\begin{align*}
	K := \bigg\{z : |z| \leq \frac{1+z_0}{2}\bigg\} \subset \mathbb{D}.
\end{align*} By the hypothesis, there exist a subset $E \subset \mathbb{C}$ consisting of three points and a constant $M > 0$ such that, for a fixed integer $k > 2$, the following inequality holds:
\begin{align*}
	\left(f^{(j)}\right)^{\#}(z) \leq M, \quad j = 0, 1, \dots, k-1,
\end{align*}
for all $z \in K \cap f^{-1}(E)$ and all $f \in \mathcal{F}$. \vspace{2mm}
	
Let $a \in E$ be an arbitrary point. For any zero $\xi_0$ of $g - a$, Hurwitz’s theorem implies the existence of a sequence $\xi_m \to \xi_0$ such that, for sufficiently large $m$,
\begin{align*}
	f_m(z_m + \rho_m\xi_m) = a.
\end{align*}
Consequently, the points $z_m + \rho_m\xi_m$ belong to $K \cap f_m^{-1}(E)$ for all $m$ sufficiently large. By the hypothesis that $k > 2$, it follows that
\begin{align*}
	\frac{\left|f_m^{(i+1)}(z_m + \rho_m\xi_m)\right|} {1+\left|f_m^{(i)}(z_m + \rho_m\xi_m)\right|^2} \leq M, \quad i = 0, 1, \ldots, k-1.
\end{align*}
In particular, for $i=0$, we obtain
\begin{align*}
	\frac{|g^{\prime}_m(\xi_m)|}{1+|g_m(\xi_m)|^2} = \rho_m \frac{|f^{\prime}_m(z_m + \rho_m\xi_m)|}{1+|f_m(z_m + \rho_m\xi_m)|^2} \leq \rho_m M,
\end{align*}
which holds for all sufficiently large $m$. From \eqref{eq-1.1} we have 
\begin{align*}
	g^{\#}(\xi_0)=\lim_{m\rightarrow\infty}\left(g_m\right)^{\#}(\xi_m) =0
\end{align*} \textit{i.e.}, $g^{\prime}(\xi_0)=0$. Since $g(\xi_0) = a$ and $g'(\xi_0) = \dots = g^{(k)}(\xi_0) = 0$, a standard application of the Taylor expansion shows that $\xi_0$ is an $a$-point of $g$ with multiplicity at least $k+1 $. Indeed, if we assume otherwise, then $g'(\xi_0) \neq 0$. However, the previous inequality implies $g'(\xi_0) = 0$, which is a contradiction. Furthermore, if $g$ were a non-constant polynomial of degree at most $1$, say $g(z) = bz + c$ for some $b, c \in \mathbb{C}$ with $b \neq 0$, then $g'(\xi_0) = b$ would be non-zero, again contradicting the fact that $g'(\xi_0) = 0$. Consequently, $g$ must be a constant or possess higher-order zeros. \vspace{1.5mm}
	
On the other hand, we observe that
\begin{align*}
	|f^{\prime}_m(z_m + \rho_m\xi_m)|&= (1+|f_m(z_m + \rho_m\xi_m)|^2) f^{\#}_m(z_m + \rho_m\xi_m)\\&\leq M\left(1+\max_{b\in E} |b|^2 \right):=B_1.
\end{align*}
It follows that for all sufficiently large $m$,
\begin{align*}
	\frac{|g^{\prime\prime}_m(\xi_m)|}{1+|g^{\prime}_m(\xi_m)|^2} & =\rho^2_m \frac{|f^{\prime\prime}_m(z_m + \rho_m\xi_m)|}{1+ \rho^2_m |f^{\prime}_m(z_m + \rho_m\xi_m)|^2} \\&= \rho^2_m \frac{|f^{\prime\prime}_m(z_m + \rho_m\xi_m)|}{1+ |f^{\prime}_m(z_m + \rho_m\xi_m)|^2}\cdot \frac{1+ |f^{\prime}_m(z_m + \rho_m\xi_m)|^2}{1+ \rho^2_m |f^{\prime}_m(z_m + \rho_m\xi_m)|^2}\\&\leq M \rho^2_m \frac{1+ |f^{\prime}_m(z_m + \rho_m\xi_m)|^2}{1+ \rho^2_m |f^{\prime}_m(z_m + \rho_m\xi_m)|^2} \\&\leq M \rho^2_m \left(1+ |f^{\prime}_m(z_m + \rho_m\xi_m)|^2 \right)\\&\leq M \rho^2_m\left(1+ B^2_1\right).
\end{align*}
Consequently, we obtain $(g')^{\#}(\xi_0) = \lim_{m \to \infty} (g'_m)^{\#}(\xi_m) = 0$, which implies $g''(\xi_0) = 0$. Given that $g(\xi_0) = a$ and $g'(\xi_0) = 0$ (as established previously), it follows that $\xi_0$ is an $a$-point of $g$ with multiplicity at least $3$. Indeed, if we assume otherwise, then $g'''(\xi_0) \neq 0$. In this case, $g$ would be a non-constant polynomial of degree at most $2$.\vspace{2mm}
	
\noindent{\bf Case A.} If $g(z)=bz+c$, for some $b,c\in\mathbb{C}$, then $g^{\prime}(\xi_0)= b\neq 0$, this is a contradiction.\\
	
\noindent{\bf Case B.} If $g(z)=b_2 z^2 + b_1 z+c_1$, for some $b_2, b_1,c_1\in\mathbb{C}$, then $g^{\prime\prime}(\xi_0)=2b_2\neq 0$, this contradicts the fact that $g^{\prime\prime} (\xi_0)=0$. \vspace{1.5mm}
	
Moreover, we see that
\begin{align*}
	|f^{\prime\prime}_m(z_m + \rho_m\xi_m)|&= (1+|f^{\prime}_m(z_m + \rho_m\xi_m)|^2) \left(f^{\prime}_m \right)^{\#}(z_m + \rho_m\xi_m)\\&\leq M\left(1+B^2_1 \right):=B_2.
\end{align*}
For all $m$ sufficiently large, we have
\begin{align*}
	\frac{|g^{\prime\prime\prime}_m(\xi_m)|}{1+|g^{\prime\prime}_m(\xi_m)|^2} & =\rho^3_m \frac{|f^{\prime\prime\prime}_m(z_m + \rho_m\xi_m)|}{1+ \rho^4_m |f^{\prime\prime}_m(z_m + \rho_m\xi_m)|^2} \\&= \rho^3_m \frac{|f^{\prime\prime\prime}_m(z_m + \rho_m\xi_m)|}{1+ |f^{\prime\prime}_m(z_m + \rho_m\xi_m)|^2}\cdot \frac{1+ |f^{\prime\prime}_m(z_m + \rho_m\xi_m)|^2}{1+ \rho^4_m |f^{\prime\prime}_m(z_m + \rho_m\xi_m)|^2}\\&\leq M \rho^3_m \frac{1+ |f^{\prime\prime}_m(z_m + \rho_m\xi_m)|^2}{1+ \rho^4_m |f^{\prime\prime}_m(z_m + \rho_m\xi_m)|^2} \\&\leq M \rho^3_m \left(1+ |f^{\prime\prime}_m(z_m + \rho_m\xi_m)|^2 \right)\\&\leq M \rho^3_m\left(1+ B^2_2\right).
\end{align*}
Consequently, we obtain $(g'')^{\#}(\xi_0) = \lim_{m \to \infty} (g''_m)^{\#}(\xi_m) = 0$, which implies $g'''(\xi_0) = 0$. Since $g(\xi_0) = a$ and $g'(\xi_0) = g''(\xi_0) = 0$, it follows that $\xi_0$ is an $a$-point of $g$ with multiplicity at least $4$. Indeed, if we assume otherwise, then $g^{(iv)}(\xi_0) \neq 0$. In this case, $g$ would be a non-constant polynomial of degree at most $3$.\vspace{1.2mm}
	
If $g(z)=d_3 z^3 + d_2 z^2 +d_1 z+d_0$, for some $d_0, d_1, d_2, d_3, \in\mathbb{C}$, then $g^{\prime\prime\prime}(\xi_0)=6d_3\neq 0$, this contradicts the fact that $g^{\prime\prime\prime} (\xi_0)=0$. \vspace{1.2mm}
	
By repeating this process, it follows that
\begin{align*}
	|f^{(k-1)}_m(z_m + \rho_m\xi_m)|= (1+|f^{(k-2)}_m(z_m + \rho_m\xi_m)|^2) \left(f^{(k-2)}_m \right)^{\#}(z_m + \rho_m\xi_m)\leq B_{(k-1)},
\end{align*}
where 
\begin{align*}
	\begin{cases}
		B_{(k-1)}:=M\left(1+ B^2_{(k-2)}\right)\vspace{1.5mm}\\ B_{(k-2)}:=M\left(1+ B^2_{(k-3)}\right)\vspace{1.5mm}\\ \cdots \cdots \cdots\cdots\cdots\cdots \cdots\cdots\vspace{1.5mm} \\B_3:=M\left(1+ B^2_2\right)\vspace{1.5mm}\\ B_2:=M\left(1+ B^2_1\right) \vspace{1.5mm}\\ B_1:= M\left(1+\max_{b\in E} |b|^2 \right).
	\end{cases}
\end{align*}
For sufficiently large $m$, we have
\begin{align*}
	\frac{|g^{(k)}_m(\xi_m)|}{1+|g^{(k-1)}_m(\xi_m)|^2} & =\rho^k_m \frac{|f^{(k)}_m(z_m + \rho_m\xi_m)|}{1+ \rho^{2(k-1)}_m |f^{(k-1)}_m(z_m + \rho_m\xi_m)|^2} \\&= \rho^k_m \frac{|f^{(k)}_m(z_m + \rho_m\xi_m)|}{1+ |f^{(k-1)}_m(z_m + \rho_m\xi_m)|^2}\cdot \frac{1+ |f^{(k-1)}_m(z_m + \rho_m\xi_m)|^2}{1+ \rho^{2(k-1)}_m |f^{(k-1)}_m(z_m + \rho_m\xi_m)|^2}\\&\leq M \rho^k_m \frac{1+ |f^{(k-1)}_m(z_m + \rho_m\xi_m)|^2}{1+ \rho^{2(k-1)}_m |f^{(k-1)}_m(z_m + \rho_m\xi_m)|^2} \\&\leq M \rho^k_m \left(1+ |f^{(k-1)}_m(z_m + \rho_m\xi_m)|^2 \right)\\&\leq M \rho^3_m\left(1+ B^2_{(k-1)} \right).
\end{align*}
Thus, it is easy to see that 
\begin{align*}
	\left(g^{(k-1)}\right)^{\#}(\xi_0) =\lim_{m\rightarrow \infty}\left(g^{(k-1)}_m\right)^{\#}(\xi_m) =0
\end{align*} \textit{i.e.}, $g^{(k)} (\xi_0)=0$. Then, it is easy to see that $\xi_0$ is an $a$-point of $g$ with multiplicity at least $k+1$. Indeed, otherwise $g^{(k+1)}(\xi)=0$. Then $g$ is a non-constant polynomial with degree at most $k$.\vspace{1.2mm}
	
If $g(z)=r_k z^k +\cdots+ r_2 z^2 +r_1 z+r_0$, for some $r_0, r_1, r_2,\ldots, r_3, \in\mathbb{C}$, then $g^{(k)}(\xi_0)=k!\neq 0$, this contradicts the fact that $g^{(k)} (\xi_0)=0$. \vspace{1.2mm}
	
Consequently, every zero of $g - a$ must have multiplicity at least $k+1$ for each $a \in E$. Let $E = \{a_1, a_2, a_3\}$ denote the three distinct points in the set. By an application of the First and Second Main Theorems of Nevanlinna theory, we obtain
\begin{align*}
	T\left(r,g\right)&\leq \sum_{j=1}^{3} \overline{N}\left(r, \frac{1}{g-a_j}\right) + o\left(T\left(r,g\right)\right) \\&\leq \frac{1}{k+1} \sum_{j=1}^{3} N\left(r, \frac{1}{g-a_j}\right) + o\left(T\left(r,g\right)\right) \\&\leq \frac{3}{k+1} T(r,g) + o\left(T\left(r,g\right)\right)
\end{align*}
for all $r\in[1,\infty)$ excluding a set of finite Lebesgue measure. Since $k > 2$, this yields a contradiction. Thus, $\mathcal{F}$ is normal in $\mathbb{D}$, which completes the proof.
\end{proof}
\begin{proof}[\bf Proof of Theorem \ref{th-2.2}]
First we consider $D$ as an unit disk $ \mathbb{D} $. On contrary, we suppose that the family $ \mathcal{F} $ be not normal at $ z_0\in\mathbb{D} $. Then by the Lemma \ref{lem-2.1}, there exist 
\begin{enumerate}
	\item[(i)] a number  $ r $, $ 0<r<1 $;
	\item[(ii)] points $ z_m $ with $ |z_m|< r \mbox{,}\;\; z_m\rightarrow z_0 $;
	\item[(iii)] functions $ f_m \in \mathcal{F} $;
	\item[(iv)] positive numbers $ \rho_m\rightarrow 0^{+} $
\end{enumerate}
such that
\begin{equation}\label{eq-2.1}
	g_m(\xi):=\frac{f_m(z_m + \rho_m\xi)}{\rho^{\alpha}_m}\rightarrow g(\xi)
\end{equation}
where 
\begin{align*}
	\alpha=\frac{D_M}{d_M}=\frac{n_1+2n_2+\cdots+kn_k}{n_0+n_1+\cdots+n_k}
\end{align*} locally uniformly with respect to spherical metric, to a non-constant meromorphic function $ g $ on compact subset of $ \mathbb{C} $, with zeros of multiplicity at least $ k $. Therefore, for $ k\in \mathbb{N} $, it is easy to see that
\begin{equation}\label{eq-2.2}
	g^{(k)}_m \rightarrow g^{(k)} \;\;\;\;\mbox{as}\;\;\;\; m\rightarrow \infty
\end{equation}
uniformly on compact subset of $ \mathbb{C}\setminus P $, where $ P $ is the set of poles of $ g $. A simple computation shows that
\begin{align}\label{eq-2.3}
	\nonumber M[g_m(\xi)]&=[g_m(\xi)]^{n_0} [g^{\prime}_m(\xi)]^{n_1}\cdots[g^{(k)}_m(\xi)]^{n_k}\\&\nonumber=\left(\frac{f_m(z_m+\rho_m\xi)}{\rho^{\alpha}_m}\right)^{n_0} \left(\frac{\rho_m f^{\prime}_m(z_m+\rho_m\xi)}{\rho^{\alpha}_m}\right)^{n_1}\cdots\left(\frac{\rho^{k}_m f^{(k)}_m(z_m+\rho_m\xi)}{\rho^{\alpha}_m}\right)^{n_k}\\&\nonumber=\frac{\left(f_m(z_m+\rho_m\xi)\right)^{n_0}}{\rho^{\alpha n_0}_m}.\frac{\left(f^{\prime}_m(z_m+\rho_m\xi)\right)^{n_1}}{\rho^{(\alpha - 1)n_1}_m} \cdots\frac{\left(f^{(k)}_m(z_m+\rho_m\xi)\right)^{n_k}}{\rho^{(\alpha - k)n_k}_m}\\&\nonumber= \frac{M[f_m(z_m+\rho_m\xi)]}{\rho^{\alpha {n_0} + (\alpha -1){n_1} + \cdots + (\alpha -k){n_k}}_m}\\&\nonumber= \frac{M[f_m(z_m+\rho_m\xi)]}{\rho^{\alpha(n_0 + n_1 + \cdots + n_k)-(n_1 + 2n_2 + \cdots + kn_k)}_m}\\&\nonumber= \frac{M[f_m(z_m+\rho_m\xi)]}{\rho^{\alpha d_M-D_M}_m}\\&= M[f_m(z_m+\rho_m\xi)].
\end{align}
By virtue of \eqref{eq-2.1}, \eqref{eq-2.2}, and \eqref{eq-2.3}, we conclude that 
\begin{align}\label{eq-2.4}
	M[g_m(\xi)]\;\;\rightarrow \;\;M[g(\xi)]
\end{align}
uniformly on compact subsets of $\mathbb{C}\setminus P$.\vspace{2mm}
	
We first demonstrate that the differential monomial $M[g]$ is non-constant and that all zeros of $M[g] - a$ possess a multiplicity of at least $2$. To show that $M[g]$ is non-constant, we note that since $g$ is non-constant and its zeros have multiplicity at least $k$, the $k$-th derivative $g^{(k)}$ cannot vanish identically. Thus, $M[g] \not\equiv 0$. Suppose, toward a contradiction, that $M[g]$ is constant. Given $n \geq 1$, it follows that $g$ must be a nowhere vanishing entire function. According to Remark \ref{rem-2.2}, $g$ takes the form $g(\xi) = e^{c\xi + d}$ for some $c \neq 0$. A direct calculation then yields
\begin{align*}
	M[g(\xi)]&=\left(g(\xi)\right)^{n_0} \left(g^{\prime}(\xi)\right)^{n_1}\cdots\left(g^{(k)}(\xi)\right)^{n_k}\\&=\left(e^{c\xi+d}\right)^{n_0} \left(ce^{c\xi+d}\right)^{n_1} \cdots \left(c^k e^{c\xi+d}\right)^{n_k} \\&= c^{n_1 +2 n_2 +\cdots + kn_k} \;e^{(c\xi +d)(n_0+n_1+\cdots+n_k)}\\&=c^{D_M}e^{(c\xi+d)d_M}.
\end{align*} 
Since $M[g(\xi)]$ is assumed to be constant, it follows that $c^{D_M}e^{d_M(c\xi+d)}$ must also be constant. However, as $c \neq 0$ and $d_M \geq 1$, this yields a contradiction. Consequently, we conclude that $M[g]$ is non-constant.\vspace{2mm}
	
Let 
\begin{align*}
	H := \bigg\{z : |z| \leq \frac{1+|z_0|}{2}\bigg\}
\end{align*} denote a compact subset of $\mathbb{D}$. By the hypothesis, there exist a value $a \in \mathbb{C} \setminus \{0\}$ and a positive constant $M^*$ such that $(M[f(z)])^{\#} \leq M^*$ for all $f \in \mathcal{F}$ and all $z \in H \cap (M[f])^{-1}(a)$. We shall demonstrate that all zeros of $M[g] - a$ possess a multiplicity of at least $2$.\vspace{2mm}
	
For any zero $\xi_0$ of $M[g] - a$, Hurwitz’s theorem implies the existence of a sequence $\{\xi_m\}_{m=1}^{\infty}$ converging to $\xi_0$ such that $M[f_m(z_m + \rho_m\xi_m)] = a$. For sufficiently large $m$, it follows that $z_m + \rho_m\xi_m \in H \cap (M[f_m])^{-1}(a)$. Thus, by the hypothesis, we obtain
\begin{align}\label{eq-2.5}
	\left(M[f_m(z_m+\rho_m\xi_m)]\right)^{\#}=\frac{|M^{\prime}[f_m(z_m+\rho_m\xi_m)]|}{1+|M[f_m(z_m+\rho_m\xi_m)]|^2}\leq M^{*}.
\end{align}
For all $ m $ sufficiently large, in view of \eqref{eq-2.5}, an easy computation shows that
\begin{align}\label{eq-2.6}
	\left(M[g_m(\xi_m)]\right)^{\#}=\frac{|M[g_m(\xi_m)]|^{\prime}}{1+|M[g_m(\xi_m)]|^2}=\frac{\rho_m|M[f_m(z_m+\rho_m\xi_m)]|^{\prime}}{1+|M[f_m(z_m+\rho_m\xi_m)]|^2}\leq \rho_m M^{*}.
\end{align}
Using \eqref{eq-2.4} and \eqref{eq-2.6}, we see that
\begin{align*}
	\left(M[g(\xi_0)]\right)^{\#}=\lim_{m\rightarrow\infty}\left(M[g_m(\xi_m)]\right)^{\#}=0.
\end{align*}
Since $(M[g(\xi_0)])^{\#} = 0$, it follows that $(M[g])'(\xi_0) = 0$. Consequently, $\xi_0$ is an $a$-point of $M[g]$ with multiplicity at least $2$. Applying Lemma \ref{lem-2.3} to $P = M[g]$, with $\theta(P) = D_M$ and $d(P) = d_M$, a direct calculation yields
\begin{align}\label{eq-2.7}
	T(r,g)&\leq\frac{D_M +1}{d_M -1}\overline{N}\left(r,\frac{1}{g}\right) + \frac{1}{d_M -1} \overline{N}\left(r,\frac{1}{M[g]-a}\right) + S(r,g)\\&\nonumber\leq\frac{D_M +1}{k(d_M -1)} N\left(r,\frac{1}{g}\right) + \frac{1}{2(d_M -1)} N\left(r, \frac{1}{M[g] -a}\right) + S(r,g)\\&\nonumber\leq\frac{D_M +1}{k(d_M -1)} T(r,g) + \frac{1}{2(d_M -1)} T(r,M[g]) + S(r,g)\\&\nonumber\leq\frac{D_M +1}{k(d_M -1)} T(r,g) + \frac{1}{2(d_M -1)} T\left(r,(g)^{n_0} (g^{\prime})^{n_1}\cdots\left(g^{(k)}\right)^{n_k}\right)\\&\quad\nonumber + S(r,g)\\&\nonumber\leq\frac{1}{2(d_M -1)}\bigg(T(r,g^{n_0}) + T(r,(g^{\prime})^{n_1}) +\cdots+ T\left(r,\left(g^{(k)}\right)^{n_k}\right)\bigg)\\&\quad\nonumber+\frac{D_M +1}{k(d_M -1)} T(r,g)+ S(r,g)\\&\nonumber\leq\frac{D_M +1}{k(d_M -1)} T(r,g) + \frac{1}{2(d_M -1)}\left(n_0 + 2n_1 + \cdots + (k+1)n_k\right)T(r,g)\\&\quad\nonumber+ S(r,g)\\&\nonumber\leq \frac{1}{2(d_M -1)}\left((n_0 + n_1 +\cdots + n_k)+(n_1 + 2n_2 +\cdots+ kn_k)\right)T(r,g)\\&\quad\nonumber +\frac{D_M +1}{k(d_M -1)} T(r,g)+ S(r,g) \\&\nonumber\leq\frac{D_M +1}{k(d_M -1)} T(r,g) + \frac{1}{2(d_M -1)}\left(d_M + D_M\right) T(r,g) + S(r,g)\\&\nonumber\leq\frac{1}{d_M -1}\left(\frac{D_M +1}{k} + \frac{d_M +D_M}{2}\right) T(r,g) + S(r,g)\\&\leq\frac{(k+2)D_M + k d_M +2}{2k(d_M -1)}\; T(r,g) + S(r,g)\nonumber
\end{align}
for all $ r\in[1,\infty) $ excluding a set of finite Lebesgue measure.\vspace{1.2mm}

\noindent It follows from the inequality \eqref{eq-2.7} that
\begin{align*}
	\left(1- \frac{(k+2)D_M + k d_M +2}{2k(d_M -1)}\right)\leq 0,
\end{align*}
which is equivalent to
\begin{align*}
	d_M\leq\frac{(k+2)D_M + 2(k+1)}{k}.
\end{align*}
This leads to a contradiction, since \begin{align*}
	d_M>\frac{(k+2)D_M + 2(k+1)}{k}.
\end{align*} Consequently, $\mathcal{F}$ is a normal family, which completes the proof.
\end{proof}
\section{\bf Concluding remarks}
It is worth noting that the extension of the results established in this paper to the setting of several complex variables presents an intriguing and non-trivial open problem. In the one-dimensional case, we have demonstrated that the degree-to-weight ratio of the differential monomial $M[f]$ serves as a sharp, intrinsic boundary for the normality of a family $\mathcal{F}$. This threshold, defined by the relationship between $d_M, D_M,$ and the order $k$, essentially quantifies the amount of `algebraic pressure' required to suppress the growth of the spherical derivative $(M[f])^{\#}$ at its $a$-points.\vspace{1.2mm}

In higher dimensions, particularly for holomorphic mappings $f: \mathbb{C}^n \to \mathbb{P}^N(\mathbb{C})$, the concept of normality is inextricably linked to the geometry of hyperplanes and the behavior of the relevant Jacobian. The interaction between the spherical derivative and the distribution of pre-images becomes significantly more complex due to the existence of exceptional sets and the requirement for hyperplanes to be in general position. Our results suggest that a similar threshold to the one found in Theorem \ref{th-2.2} may exist for holomorphic mappings, potentially involving the intersection multiplicities of shared hyperplanes and the weights of associated differential operators.\vspace{1.2mm}

Furthermore, our extension of Lappan's theorem to a three-point set $E$ via higher-order spherical derivatives offers a new perspective on the minimum cardinality requirements in higher dimensions. While classical results for mappings into $\mathbb{P}^N(\mathbb{C})$ often necessitate $2N+1$ or more hyperplanes to guarantee normality, our work hints that the utilization of derivatives $f^{(j)}$ might allow for a reduction in the number of targets, provided the spherical growth of these derivatives is controlled.\vspace{1.2mm}

Our framework thus provides a foundational bridge that may complement the higher-dimensional analogs explored by Tu \cite{Tu-PAMS-1999} and Yang et al. \cite{Yang-Liu-Pang-HJM-2016}. Specifically, the techniques developed here for managing the rescaling process in the presence of differential monomials could prove instrumental in addressing normality criteria for holomorphic mappings where the standard Picard-type values are scarce. Future research into the multidimensional Zalcman-type rescaling lemma in conjunction with differential monomials could yield a more unified theory of normality across these disparate dimensions.

\vspace{2mm}

\noindent{\bf Acknowledgment:} The authors would like to thank the anonymous referees for their helpful comments and suggestions that led to an improved version of the manuscript.\\

\noindent{\bf Author Contributions:} All authors actively worked on the research contained in the paper. All authors reviewed the manuscript.\\

\noindent\textbf{Compliance of Ethical Standards:}\\

\noindent\textbf{Conflict of interest.} The authors declare that there is no conflict  of interest regarding the publication of this paper.\vspace{2mm}

\noindent\textbf{Data availability statement.}  Data sharing is not applicable to this article as no datasets were generated or analyzed during the current study. \vspace{2mm}

\noindent\textbf{Declaration of Fundings.} The authors declare that no funds, grants, or other support were received during the preparation of this manuscript.


\begin{thebibliography}{99}	
	
	\bibitem{Ahamed-Mandal-MM-2022} {\sc M. B. Ahamed} and {\sc S. Mandal}, Certain properties of normal meromorphic and normal harmonic mappings,\textit{ Monatshefte für Mathematik}, \textbf{200} (2023), 719–736. 
	
	\bibitem{Ahamed-Mandal-CMFT-2024} {\sc M. B. Ahamed} and {\sc S. Mandal}, Normality criterion concerning total derivatives of holomorphic functions in $\mathbb{C}^n$, \textit{Comput. Methods Funct. Theory} (2024).
	
	\bibitem{Aladro-Krantz-JMAA-1991} {\sc G. Aladro} and {\sc S. G. Krantz}, A criterion for normality in $ \mathbb{C}^n $, \textit{J. Math. Anal. Appl.} \textbf{161} (1991), 1–8.
	
	\bibitem{Arbe-Hern-MM-2019} {\sc H. Arbel$ \acute{a} $ez}, {\sc R. Hern$\acute{a}$ndez} and {\sc W. Sierra}, Normal harmonic mappings, \textit{ Monatshefte für Mathematik} \textbf{190} (2019), 425-439.
	
	\bibitem{Chuang-WP-1987}{\sc C. T. Chuang}, On differential polynomials. In: Analysis of One Complex Variable, \textit{World Sci.} (1987), 12–32.
	
	\bibitem{Clunie-Hayman-CMH-1966} {\sc J. Clunie} and {\sc W.K. Hayman}, The spherical derivative of integral and meromorphic functions, \textit{Comment. Math. Helv.} \textbf{40} (1966), 117-148.
	
	\bibitem{Deng-Ponn-Qiao-MM-2020} {\sc H. Deng}, {\sc S. Ponnusamy} and {\sc J. Qiao}, Properties of normal harmonic mappings, \textit{ Monatshefte für Mathematik} \textbf{193} (2020):605-621.
	
	\bibitem{Dethloff-Tan-Thin-JMAA}{\sc G. Dethloff}, {\sc T.V. Tan} and {\sc N.V. Thin}, Normal criteria for families of meromorphic functions, \textit{J. Math. Anal. Appl.} \textbf{411} (2014), 675–683.	
	
	\bibitem{Doeringer-PJM-1982} {\sc W. Doeringer}, Exceptional values of differential polynomials, \textit{Pac. J. Math.} \textbf{98} (1982).
	
	\bibitem{Dovbush-JGA-2022} {\sc P.V. Dovbush}, On a Normality Criterion of W. Schwick, \textit{J. Geom. Anal.} \textbf{31}(2021), 5355–5358.
	
	\bibitem{Dovbush-CVEE-2022} {\sc P. V. Dovbush}, On normal families in $ \mathbb{C}^n $, \textit{Complex Var. Elliptic Equ.} \textbf{67}(1)(2022), 1-8.
	
	\bibitem{Hayman-CP-1964}{\sc W.K. Hayman}, Meromorphic Functions, \textit{Clarendon Press, Oxford} (1964).
	
	\bibitem{Hamymen-Storvick-BLMS-1971} {\sc W. K. Hayman} and {\sc D. A. Storvick}, On normal functions, \textit{Bull. Lond. Math. Soc.} \textbf{3} (1971), 193-194.
	
	\bibitem{Hinchliffe-CMFT-2002}{\sc J.D. Hinchliffe}, On a result of Chuang retated to Hayman’s alternative, \textit{Comput. Methods Funct. Theory} \textbf{2} (2002), 293–297.
	
	\bibitem{Hinkkanen-NZJM-1993} {\sc A. Hinkkanen}, Normal family and Ahlfors’s five islands theorem,\textit{N. Z. J. Math.} \textbf{22} (1993), 39–41.
	
	\bibitem{Hu-Thin-CVEE-2020} {\sc P. C. Hu} and {\sc N. V. Thin}, Generalizations of Montel’s normal criterion and Lappan’s five-valued theorem to holomorphic curves, \textit{Complex Var. Elliptic Equ.} \textbf{65}(4) (2020), 525-543.
	
	\bibitem{Lappan-MZ-1965} {\sc P. Lappan}, Some results on harmonic normal functions, \textit{Math. Zeits.} \textbf{90}(2) (1965), 155-159.
	
	\bibitem{Lappan-CMH-1974} {\sc P. Lappan}, A criterion for a meromorphic function to be normal, \textit{Comment. Math. Helv.} \textbf{49} (1974), 492-495.
	
	\bibitem{Lappan-AASF-1977} {\sc P. Lappan}, The spherical derivative and normal functions, \textit{Ann. Acad. Sci. Fenn. Ser. A} \textbf{3}(2) (1977), 301-310.
	
	\bibitem{Lappan-MA-1978} {\sc P. Lappan}, On the normality of derivatives of functions, \textit{Math. Ann.} 238(1978), 141–146.
	
	\bibitem{Lappan-RRMPA-1994} {\sc P. Lappan}, A uniform approach to normal families, \textit{Rev. Roumaine Math. Pures Appl.} \textbf{39} (1994), 691–702.
	
	\bibitem{Lehto-Virtanen-AM-1957}{\sc O. Lehto} and {\sc K. Virtanen}, Boundary beheviour and normal meromorphic functions, \textit{Acta Math.} \textbf {97} (1957), 47-65. 
	
	\bibitem{Montel-CB-1927} {\sc P. Montel}, Lecons sur les families normales de fonctions analytiques et leur applicaeions, \textit{Coll. Borel}, (1927).
	
	\bibitem{Noshiro-JFSHU-1938} {\sc K. Noshiro}, Contributions to the theory of meromorphic functions in the unit circle, \textit{J. Fac. Svi. Hokkaido Univ.} \textbf{7} (1938), 149-159.
	
	\bibitem{Pang-Zalcman-1999} {\sc X. C. Pang} and {\sc L. Zalcman}, On theorems of Hayman and Clunie, \textit{N. Z. J. Math.} \textbf{28}(1999), 71-75. 
	
	\bibitem{Pommerenke-BLMS-1972} {\sc Ch. Pommerenke}, Problems in complex function theory, \textit{Bull. Lond. Math. Soc.} \textbf{4} (1972), 354–366.
	
	\bibitem{Ru-WS-2001} {\sc M. Ru}, Nevanlinna theory and its relation to diophantine approximation, Singapore, \textit{World Scientific}, 2001.
	
	\bibitem{Tran-Nguyen-CMFT-2017} {\sc T. V. Tan} and {\sc N. V. Thin}, On Lappan's five-point theorem, \textit{Comput.Methods Funct.Theory} \textbf{17} (2017), 47-63.
	
	\bibitem{Tan-Thin-JMAA-2017} {\sc T. V. Tan}, {\sc N. V. Thin} and {\sc V. V. Truong}, On the normality criteria of Montel and Bergweiler–Langley, \textit{J. Math. Anal. Appl.} \textbf{448} (2017), 319-325.
	
	\bibitem{Tu-PAMS-1999} {\sc Z. H. Tu}, Normality criteria for families of holomorphic mappings of several complex variables into $ \mathbb{P}^N(\mathbb{C}) $, \textit{Proc. Amer. Math. Soc.} \textbf{127} (1999), 1039-1049. 
	
	\bibitem{Yamashita-MZ-1975} {\sc S. Yamashita}, On normal meromorphic functions, \textit{Math. Z.} \textbf{141} (1975), 139–145.
	
	\bibitem{Yang-SSP-1993} {\sc L. Yang}, Value distribution theory, \textit{Springer and Science Press}, Berlin, 1993.
	
	\bibitem{Yang-Fang-Pang-PJM-2014} {\sc L. Yang}, {\sc C. Y. Fang} and {\sc X. C. Pang}, Normal families of holomorphic mappings into complex projective space concerning shared hyperplanes, \textit{Pac. J. Math.}  \textbf{272} (2014), 245–256.
	
	\bibitem{Yang-Liu-Pang-HJM-2016} {\sc L. Yang}, {\sc X. J. Liu} and {\sc X. C. Pang}, On families of meromorphic map into the complex projective space, \textit{Housten. J. Math.}  \textbf{42} (2016), 775–789.
	
	\bibitem{Yang-Yi-2006} {\sc C. Yang, H. Yi}, Uniqueness theory of meromorphic functions, \textit{Science Press, Beijing} (2006).
	
	\bibitem{Yosida-PPMSJS-1934} {\sc K. Yosida}, On a class of meromorphic functions, \textit{Proc. Phys. Math. Soc. Jpn. Ser.} \textbf{16}(3) (1934), 227-235.
	
	\bibitem{Zalcman-BAMS-1998}{\sc L. Zalcman}, Normal families: new perspective, \textit{Bull. Amer. Math. Soc.} \textbf{35} (1998), 215-230. 
\end{thebibliography}
\end{document}